\documentclass{article}
\usepackage[utf8]{inputenc}
\usepackage{amsmath,amsthm,amssymb,graphicx,tikz,caption,subcaption,xcolor,nicefrac,hyperref,cleveref}
\usepackage{fullpage}
\usepackage{bm}
\usepackage{xspace}

\allowdisplaybreaks

\newtheorem{theorem}{Theorem}

\newtheorem{lemma}{Lemma}

\newtheorem{remark}{Remark}

\newcommand{\oS}{\overline{S}}

\newcommand{\N}{\mathbb{N}}
\newcommand{\R}{\mathbb{R}}

\newcommand{\eps}{\epsilon}

\newcommand{\bbC}{\mathbb{C}}
\newcommand{\bbU}{\mathbb{U}}
\newcommand{\bbUl}{\mathbb{U}_{\mathrm{local}}}

\newcommand{\bs}{\mathbf{s}}
\newcommand{\bx}{\mathbf{x}}
\newcommand{\by}{\mathbf{y}}
\newcommand{\bS}{\mathbf{S}}
\newcommand{\obs}{\overline{\bs}}
\newcommand{\obS}{\overline{\bS}}
\newcommand{\os}{\overline{s}}
\newcommand{\ox}{\overline{x}}
\newcommand{\obx}{\overline{\bx}}

\newcommand{\confv}{\mathrm{cv}}
\newcommand{\cP}{\mathcal{P}}
\newcommand{\cO}{\mathcal{O}}
\newcommand{\bbXl}{\mathbb{X}}
\newcommand{\bb}{\mathbf{b}}
\newcommand{\obb}{\overline{\bb}}
\newcommand{\ob}{\overline{b}}
\newcommand{\ov}{\overline{v}}

\newcommand{\tx}{\tilde{x}}
\newcommand{\tbx}{\tilde{\bx}}
\newcommand{\cI}{\mathcal{I}}
\newcommand{\alphaup}{\alpha^{\uparrow}}
\newcommand{\alphalo}{\alpha^{\downarrow}}
\newcommand{\gammaup}{\gamma^{\uparrow}}
\newcommand{\gammalo}{\gamma^{\downarrow}}
\newcommand{\ualpha}{\underline{\alpha}}
\newcommand{\unu}{\underline{\nu}}
\newcommand{\scr}{\mathrm{scr}}
\newcommand{\argmin}{\mathrm{argmin}}

\newcommand{\Paren}[1]{\left(#1\right)}

\title{On Expansion of Random Regular Graphs:\\ Improved Lower Bounds for Small Even Degrees}
\date{\today}

\author{
Pasin Manurangsi\thanks{Email: \texttt{pasin@google.com}.} \\
Google Research
}

\begin{document}

\maketitle

\begin{abstract}
We show that a simple scoring-based tie-breaking can help improve lower bounds for the expansion (aka isoperimetric number) of random regular graphs with small even degrees. Specifically, for degrees $4, 6$ and $8$, we show that, with high probability, the expansions are at least $0.489, 1.120$ and $1.813$ respectively. 
\end{abstract}

\section{Introduction}

Let $G = (V, E)$ be any undirected unweighted graph. For every set $S, T \subseteq V$, we use $E(S, T)$ to denote the set of edges with one endpoint in $S$ and one endpoint in $T$. Recall that the expansion of a set $S$ is defined as $\iota_G(S) := \frac{|E(S, V \setminus S)|}{\min\{|S|, |\oS|\}}$. The expansion of the graph $\iota(G)$ is the minimum expansion among all $S \subseteq V$. 

In his seminal work, Bollob{\'{a}}s~\cite{Bollobas88} introduced the so-called pairing model $\cP(n, \Delta)$ for random $\Delta$-regular graphs on $n$ vertices and showed that\footnote{The bound is stated more precisely in \Cref{thm:bollobas-classic}; see also \Cref{tab:bounds_comparison}.}, as $n \to \infty$, a random graph drawn from $\cP(n, \Delta)$ has expansion bounded away from zero with high probability for any $\Delta \geq 3$.
Several subsequent works obtained improvements to this bound for small degrees. Specifically, for $\Delta = 3$, Kostochka and Melnikov~\cite{KM92} gave a lower bound of $1/4.95 \geq 0.202$. Amit and Linial~\cite{AmitL06} devised a generic improvement for all $\Delta$ using $\eps$-nets. However, the actual numerical improvement over the original bound in \cite{Bollobas88} is quite small; the authors estimated their improvement to be $\approx$0.1\% for $\Delta = 3$. Later, Lampis~\cite{Lampis12} gave a generic ``local improvement'' technique for improving the lower bounds for all degrees $\Delta \geq 4$, although some of the claims in his preprint are not formal\footnote{E.g., the preprint asserts ``bisections are again the interesting case, so we may assume that $|S| = n/2$'' without any proof.}. A recent work by Daneshgar and Shojaedin~\cite{formalized-local} formalizes this approach.

\paragraph{Our Contributions.}
Our main results are improved lower bounds for degree 4, 6 and 8: 

\begin{theorem} \label{thm:main}
Let $\Delta \in \{4, 6, 8\}$, and let $\nu^*_{\Delta}$ be such that $H_{\Delta}(\nu^*_{\Delta}/2) = 0$ where $H_{\Delta}$ is as defined in \eqref{eq:def-for-main-bound}. In particular, $\nu^*_4 = 0.4894\dots, \nu^*_6 = 1.1205\dots, \nu^*_8 = 1.8130\dots$.

Then, for any constant $\nu < \nu^*_{\Delta}$, a graph $G \sim \cP(n, \Delta)$ satisfies $\iota(G) > \nu$ asymptotically almost surely.
\end{theorem}

We remark that, since $H_{\Delta}$ is strictly increasing (\Cref{lem:root-bound}) and $H_{\Delta}(\gamma)$ can be computed (to arbitrary precision) efficiently (\Cref{thm:min-entropy-main}), we can compute $\nu^*_\Delta$ (to arbitrary precision) using binary search.

The numerical comparison between our lower bounds and previously known lower bounds are given in \Cref{tab:bounds_comparison}. As can be seen, our improvements over previous bounds are significant: While previous improvements are less than $1.5\%$ over that of the original work by Bollobas~\cite{Bollobas88}, ours are more than $5\%$ in all cases. In fact, our improvement for $\Delta = 4$ is over $10\%$ compared to \cite{Bollobas88}.

\begin{table}[h!]
\centering % To center the table on the page
\begin{tabular}{|c|c|c|c|c|} 
\hline
\textbf{Degree $\Delta$} & \textbf{\cite{Bollobas88}} & \textbf{\cite{AmitL06}} & \textbf{\cite{formalized-local}} & \textbf{Our Bounds} \\ 
\hline
4 & 0.4401 & 0.4403 & 0.4452 & 0.4894 \\
6 & 1.0437 & 1.0438 & 1.0584 & 1.1205 \\
8 & 1.7160 & 1.7161 & 1.7297 & 1.8130 \\
%10 & 2.430 & 2.465 & 2.543 \\
%12 & 3.168 & 3.216 & 3.301 \\
%14 & 3.934 & 3.990 & 4.078 \\
\hline
\end{tabular}
\caption{Comparison of previous and our new lower bounds for expansion of regular random graphs of different degrees. The reported numbers from previous works are from Table 2 in \cite{formalized-local}. We also remark that, while \cite{formalized-local} formalizes techniques from \cite{Lampis12}, the numbers they claim are not exactly the same. We only report those in \cite{formalized-local} since they are formally proved.}
\label{tab:bounds_comparison}
\end{table}

\paragraph{Overview of Techniques.} Before we can discuss our proof overview, we recall previous work. \cite{Bollobas88} use a simple strategy of taking the union bound over all sets $S \subseteq [n]$ and calculate the probability that it violates the expansion condition, i.e. $\iota_G(S) \geq \nu$. An improvement in \cite{Lampis12,formalized-local} comes from an observation that the latter probability can be strengthened further by only including $S$ that has the minimum expansion. In a sense, these works apply a \emph{tie-breaking} rule and only count subsets that are the smallest under the rule. Generally, bounding the probability that $S$ has minimum expansion is challenging. Thus, \cite{Lampis12} only keeps $S$ that cannot be ``locally'' improved by \emph{swapping} one vertex from $S$ to $\oS$. This rules out certain ``degree profile'' (formalized as \emph{configuration vector} in \Cref{subsec:prelim-config}). E.g., it is not possible for both $S$ and $\oS$ to have vertices with strictly more than half of their neighbors outside of the set. As otherwise, we can simply swap them and reduce the expansion. This is generally the rough ideas from previous work that we will build on. 

Our work introduces the following new tie-breaking rules: %In particular, we make two specific tie-breaking in our proofs.
\begin{itemize}
\item \emph{Scoring-Based Tie-Breaking:} The most crucial addition to our proof is a scoring-based tie-breaking: We give random scores to vertices and, if there are multiple $S$ with the same expansion, only counts those that have the minimum total score. The main effect of this can be seen as follows: Local improvement strategies of \cite{Lampis12} still allow arbitrary number of vertices in $S, \oS$ to have exactly half of their neighbors outside of the set. However, we show that, if this happens, the probability that $S$ has the minimum score is tiny (\Cref{lem:lex}). This is the main source of improvement in the bounds. 
\item \emph{Local Improvement by Moving Single Vertex:} Another observation we make is that we can also make local improvements by \emph{moving a single vertex} from $S$ to $\oS$ (and vice versa). This seemingly innocuous observation leads to a more restricted degree profile than that of \cite{Lampis12,formalized-local} (see \Cref{sec:local-improvement}). In particular, for the unbalanced case where $|S| < |\oS|$, we can show that no vertex in $\oS$ can even have exactly half of its neighbors outside of the set. This actually affords us a numerical gap between the balanced case and the unbalanced case, which ultimately allows us to use a discretization approach to numerically verify the latter (\Cref{subsec:asym}). 
\end{itemize}

\paragraph{Other Related Work.}
A number of works have also studied the size of the minimum bisection (aka \emph{bisection width}) of random graphs. In particular, both upper and lower bounds have been derived for different degrees (e.g.~\cite{DiazSW07,LichevM23}). We note that upper bounds on the bisection width immediately imply upper bounds on the expansion of the graph, but the lower bounds do not. Meanwhile, our lower bounds imply lower bounds on the bisection width as well.
%Thus, the latter works are not directly comparable to our (lower) bounds.
%
Much work has also been done for the related questions for \emph{vertex} expansion; see e.g. \cite{DiazDSS24} and references therein.
\section{Preliminaries}
In this section, we recall a few preliminaries that will useful throughout.
For any positive integer $K$, let $[K] := \{1, \dots, K\}$ and $[K]_0 := \{0, \dots, K\}$. Recall also that $K!! := \prod_{i=0}^{\lfloor (K-1)/2 \rfloor} (K - 2i)$.

\subsection{Random Regular Graphs and Expansion}

Throughout this work, we assume for simplicity that both $n$ and $\Delta$ are even positive integers. We always consider the asymptotic setting where $\Delta$ is fixed and $n \to \infty$, and we will not state this explicitly henceforth.

We use the so-called pairing model\footnote{We remark that our high probability results also hold for $\Delta$-regular graphs that are drawn uniformly at random from the set of all $\Delta$-regular graphs, due to a well known connection between two models~\cite{Wormald-models}.} of \cite{Bollobas88}. Namely, we consider a distribution $\cP_{n, \Delta}$ of a random $\Delta$-regular graph on $n$ vertices generated as follows. First, start by taking a random perfect matching over $[n] \times [\Delta]$; such a matching is referred to as a \emph{configuration}. Then, for each $i \in [n]$, we merge all $\Delta$ elements $(i, 1), \dots, (i, \Delta)$ to create our final graph $G = ([n], E)$. %When we need to be precise, we write $\bbC$ to denote a configuration and $G(\bbC)$ to denote the graph resulting from the realization.
The original work of \cite{Bollobas88} gives the following lower bound on the expansion of the graph. (See \Cref{tab:bounds_comparison} for the numerical values for $\Delta \in \{4, 6, 8\}$.)

\begin{theorem}[\cite{Bollobas88}] \label{thm:bollobas-classic}
Let $\eta_{\Delta} \in (0, 1)$ be the solution to the following equation
\begin{align*}
(1 - \eta_\Delta) \log_2(1 - \eta_{\Delta}) + (1 + \eta_\Delta) \log_2(1 + \eta_{\Delta}) = \frac{4}{\Delta}
\end{align*}
and let $\unu_{\Delta} > 0$ be such that $\unu_{\Delta} < (1 - \eta_{\Delta})\frac{\Delta}{2}$. Then, a graph $G \sim \cP(n, \Delta)$ satisfies $\iota(G) > \unu_{\Delta}$ a.a.s.
\end{theorem}

It will be convenient to consider only sets $S$ that are sufficiently large, e.g. for convenience of applying Stirling's approximation.
To do so, for $u \in (0, 1/2]$, we additionally define $\iota_{\leq u}(G) = \min_{|S| \leq u \cdot n} \iota_G(S).$
Note that $\iota(G)$ coincides with $\iota_{1/2}(G)$.

The following lemma asserts that small subsets have large expansions a.a.s. This follows from \cite{KolesnikW14}. In particular, \cite[Lemma 12]{KolesnikW14} gives a lower bound $\hat{A}_{\Delta}(u)$ on the expansion $\iota_{\leq u}(G)$ for $G \sim \cP(n, \Delta)$. The lower bound $\hat{A}_{\Delta}(u)$ is efficiently computable and, for $u = 0.1$, we evaluate it to $\hat{A}_{4}(u) \geq 0.933, \hat{A}_{6}(u) \geq 2.163, \hat{A}_{6}(u) \geq 3.507$. Below we state a very loose bound which is nevertheless sufficient for us. %We note however that a standard technique (directly from \cite{Bollobas88}) is already sufficient to prove such a statement, but with perhaps a worse bound on $\eps$. (For our proof, any absolute values $\eps > 0$ suffice.)

\begin{theorem}[\cite{KolesnikW14}] \label{thm:small-set}
For $\Delta \in \{4, 6, 8\}$ and $\ualpha = 0.1$, $G \sim \cP(n, \Delta)$ satisfies $\iota_{\leq \ualpha}(G) \geq \frac{3\Delta}{8}-1$ a.a.s. 
\end{theorem}

\subsection{Configuration Vector}
\label{subsec:prelim-config}

We use the notion of configuration vector from \cite{Lampis12}. For a partition $(S, \oS)$ of $[n]$ and a graph $G = ([n], E)$, the \emph{cross-degree} of $v \in S$ (resp. $v \in \oS$) is the number of neighbors of $v$ in $\oS$ (resp. in $S$). We write $\bS = (S_1, \dots, S_\Delta)$ (resp. $\obS = (\oS_1, \dots, \oS_\Delta)$) where $S_j \subseteq S$ (resp. $\oS_j \subseteq \oS$)  denotes the set of vertices in $S$ (resp. $\oS$) with cross-degree $j$. Finally, the \emph{configuration vector} of $(S, \oS)$, denoted by $\confv_G(S)$, is a pair $(\bs, \obs)$ of vectors $\bs = (s_1, \dots, s_\Delta)$ and $\obs = (\os_1, \dots, \os_\Delta)$ where $s_j = |S_j|, \os_j = |\oS_j|$.

%Namely, for a partition $(S, \oS)$ of $[n]$ and a graph $G = ([n], E)$, we write a vector $\bs = (s_1, \dots, s_\Delta)$ and $\obs = (\os_1, \dots, \os_\Delta)$ where $s_j$ (resp. $\os_j$) denote the number of vertices in $S$ (resp. $\oS$) that that has $j$ of its neighbors in $\oS$ (resp. $S$). We refer to $(\bs, \obs)$ as the \emph{configuration vector} of $(S, \oS)$ w.r.t. $G$, denoted by $\confv_G(S)$. %Furthermore, 

Let $\bbU(k, c)$ denote the set of all configuration vectors $(\bs, \obs)$ such that $\sum_{i = 0}^{\Delta} i s_i = \sum_{i = 0}^{\Delta} i \os_i = c, \sum_{i = 0}^{\Delta} s_i = k$ and $\sum_{i = 0}^{\Delta} \os_i = n - k$. The following (simple) formula will be convenient for our proofs.

\begin{lemma}[\cite{Lampis12}] \label{lem:lampis-counting}
For any fixed partition $(S, \oS)$ with $|S| = k$ and any $(\bs, \obs) \in \bbU(k, c)$, the probability over $G \sim \cP(n, \Delta)$ 
%the number of configurations $\bbC$ such 
that $\confv_{G(\bbC)}(S) = (\bs, \obs)$ is
\begin{align*}
P(\bs, \obs) := \frac{\left(k!(n - k)!\right) \cdot \left(c! (\Delta k - c)!! (\Delta(n- k) - c)!!\right)}{(\Delta n)!!} \cdot \left(\prod_{i=0}^{\Delta} \frac{1}{s_i!} \cdot \binom{\Delta}{i}^{s_i} \right) \cdot \left(\prod_{i=0}^{\Delta} \frac{1}{\os_i!} \cdot \binom{\Delta}{i}^{\os_i} \right)
\end{align*}
\end{lemma}

\subsection{Stirling's Approximation}
\label{subsec:stirling}

We will use Stirling's approximation to aid our proofs throughout. In particular, recall the following:
\begin{align*}
\ln n! = n(\ln n - 1)  + \Theta(1),
&& \ln n!! = \frac{n}{2}(\ln n - 1)  + \Theta(1).
\end{align*}

\subsection{Minimizing Relative Entropy Subject to Linear Constraints}

In our proofs, a type of optimization problem will show up repeatedly; we define and solve it below.

Let $T \in \N$.
For $a, c \in \R, \bb \in \R^{[T]_0}$. Consider the following optimization problem:\footnote{Throughout, we assume that $a, c > 0$ and $b_i > 0$ for all $i \in [T]_0$.}
\begin{align}
\Phi^*(a, c, \bb) &:= \max_{\by \in \R^{[T]_0}} \sum_{i \in [T]_0} y_i \ln(b_i / y_i) \label{opt-prob:entropy-min}  \\
\text{subject to }  & \sum_{i \in [T]_0} y_i = a \nonumber \\
&\sum_{i \in [T]_0} i \cdot y_i = c \nonumber \\
& y_i \geq 0 & & \forall i \in [T]_0 \nonumber
\end{align}
Note that, with appropriate normalization, this problem is simply the problem of minimizing relative entropy with respect to a mean constraint. Nevertheless, we forgo the normalization here since this generic formulation is more convenient for our subsequent applications.

The following lemma gives a simple-to-use formula for the solution to the problem together with its derivative with respect to $c$, which will become handy in our calculations later on.

\begin{theorem} \label{thm:min-entropy-main}
Consider the optimization problem \eqref{opt-prob:entropy-min} under the assumption that $c < a \cdot T$.

For every $z \in \R_+$, let $Z(z; \bb) = \sum_{i \in [T]_0} b_i z^i$. The optimal solution $\by^*$ to \eqref{opt-prob:entropy-min} is given by
\begin{align*}
    y_i^* = b_i \cdot a \cdot \frac{(z^*)^i}{Z(z^*)} & &\forall i \in [T]_0,
\end{align*}
where $z^*$ is the unique\footnote{The uniqueness of the root follows from Descartes' rule of signs since $c < a \cdot T$.} positive solution to the  equation $0 = \Paren{\sum_{i \in [T]_0} (ai - c)b_i z^i}$.

The total derivative of the optimal objective value with respect to the parameter $c$ is %equivalent to the Lagrange multiplier $\mu = -\ln z^*$, specifically:
\begin{align} \label{eq:derivative-min-rel-ent}
    \frac{d \Phi^*(a, c, \bb)}{d c} = -\ln z^*.
\end{align}
\end{theorem}

\begin{proof}
Let $G(\by) := \sum_{i \in [T]_0} y_i \ln(b_i / y_i)$. This function is strictly concave in $\R_{+}^{[T]_0}$ (its second derivative with respect to $y_i$ is $-1/y_i < 0$). Thus, the stationary point found via the method of Lagrange multipliers will correspond to the global maximum of $G(\by)$ on the interior of the feasible region. 

We formulate the Lagrangian $\mathcal{L}(\by, \lambda, \mu)$ by relaxing the equality constraints as follows\footnote{We ignore the non-negativity constraint $y_i \geq 0$; the solution obtained will obviously satisfy this constraint.}: 
\begin{align*}
    \mathcal{L}(\by, \lambda, \mu) = \sum_{i} y_i \ln\left(\frac{b_i}{y_i}\right) - \lambda \left( \sum_{i} y_i - a \right) - \mu \left( \sum_{i} i \cdot y_i - c \right)
\end{align*}
where $\lambda$ and $\mu$ are the Lagrange multipliers. Taking the partial derivative of $\mathcal{L}$ with respect to $y_k$ and setting it to zero yields:
\begin{align*}
    \frac{\partial \mathcal{L}}{\partial y_k} = \ln\left(\frac{b_k}{y_k}\right) + y_k \left( -\frac{1}{y_k} \right) - \lambda - k \mu = 0 & &\Rightarrow y_k = b_k e^{-1 - \lambda} (e^{-\mu})^k. %\\
    %&= \ln\left(\frac{b_k}{y_k}\right) - 1 - \lambda - k \mu = 0,
\end{align*}
%which implies $y_k = b_k e^{-1 - \lambda} (e^{-\mu})^k$
To simplify the notation, let $z^* = e^{-\mu}$. This becomes:
\begin{equation}
    y_k = b_k e^{-1 - \lambda} (z^*)^k \label{eq:yk_z_star}
\end{equation}

To determine the multiplier $\lambda$, we apply the first constraint $\sum_i y_i = a$, which gives $e^{-1 - \lambda} Z(z^*) = a$.
Substituting this back into \eqref{eq:yk_z_star} provides the optimal $y^*$:
\begin{equation}
    y_k^* = b_k \cdot a \cdot \frac{(z^*)^k}{Z(z^*)} \label{eq:y_star}
\end{equation}
Since $a, b_k > 0$, and $z^* = e^{-\mu} > 0$, it follows that $y_k^* > 0$ for all $k$.

To find the condition for $z^*$, we apply the second constraint $\sum_i i \cdot y_i = c$.
Rewriting this equation gives:
\begin{equation}
\sum_{i \in [T]_0} (ai - c)b_i (z^*)^i = 0
 \label{eq:z_constraint}
\end{equation}
This establishes the value of $z^*$ implicitly as a function of $a$ and $c$, completing the proof of the first part.

As for the derivative, we simply use the Envelope Theorem, which states that the total derivative of the optimized objective function with respect to a parameter is equal to the partial derivative of the Lagrangian with respect to that parameter, evaluated at the optimal point. This yields
\begin{align*}
\frac{d \Phi^*(a, c, \bb)}{dc} &= \frac{\partial \mathcal{L}}{\partial c}\bigg|_{y^*, \lambda, \mu} = \mu = - \ln z^*. \qedhere
\end{align*}
\end{proof}
\section{Tie-Breaking Violating Sets}

When looking for a set $S$ (with $|S| \leq n/2$) that has minimum expansion, it is crucial to use tie-breaking. Below, we discuss two techniques for tie-breaking. First, in \Cref{sec:local-improvement}, we modify the \emph{local improvement} technique from \cite{Lampis12,formalized-local}. Second, in \Cref{sec:lexicographic-tie-breaking}, we introduce a \emph{scoring-based tie-breaking}.

\subsection{Local Improvement}
\label{sec:local-improvement}

We say that $S$ is an \emph{optimal set} for $G$ if $S$ has the minimum expansion. %Note that there could be multiple optimal sets for $G$.

We will use the idea of local improvement similar to \cite{Lampis12}. However, instead of considering \emph{swapping two vertices}--one from $S$ with one from $\oS$, we instead consider moving \emph{moving one vertex}. Interestingly, this actually gives us a stronger limitation on the configuration vector for our setting compared to \cite{Lampis12}.

Specifically, let $\bbUl(k, c)$ denote the set of all $(\bs, \obs) \in \bbU(k, c)$ that satisfies the following constraints:
\begin{itemize}
\item $s_j = \os_j = 0$ for all $j > \frac{\Delta}{2}$, and,
\item if $k < n/2$, it must be that $\os_{\frac{\Delta}{2}} = 0$.
\end{itemize}

Then, we can show the following lemma.

\begin{lemma}[Local Optimality]
Let $\Delta$ be any even positive integer.
If $S$ is an optimal set for $\Delta$-regular graph $G$ such that $\iota_G(S) < 2$, then its configuration vector must belong to $\bbUl(k, c)$ for some $k, c \in \N$.
\end{lemma}

\begin{proof}
Suppose for the sake of contradiction that there is an optimal set $S$ for $G$ such that $\iota_G(S) < 2$ and $\confv_G(S) \notin \bbUl(k, c)$. We must be in one of the following two cases:
\begin{itemize}
\item Case I: $s_j > 0$ or $\os_j > 0$ for $j > \Delta/2$. Assume w.l.o.g. that it is the former. That is, there exists $v \in S$ such that $|E(\{v\}, \oS)| = j \geq \Delta/2 + 1$.  Consider instead the set $S' = S \setminus \{v\}$. Notice that we have $|E(S', \oS')| = |E(S, \oS)| + (\Delta - j) - j \leq |E(S, \oS)| - 2$, and $\min\{|S'|, |\oS'|\} \geq \min\{|S|, |\oS|\} - 1$. Thus,
\begin{align*}
\iota_G(S') = \frac{|E(S', \oS')|}{\min\{|S'|, |\oS'|\}} \leq \frac{|E(S, \oS)| - 2}{\min\{|S|, |\oS|\} - 1} = \iota_G(S) + \frac{\iota_G(S) - 2}{\min\{|S|, |\oS|\} - 1} < \iota_G(S),
\end{align*}
where the last inequality follows from our assumption $\iota_G(S) < 2$.
\item Case II: $k < n/2$ and $\os_{\Delta/2} > 0$. This means that there exists $v \in S$ such that $|E(\{v\}, \oS)| = \Delta/2$. Again, consider $S' = S \setminus \{v\}$. In this case, we have $|E(S', \oS')| = |E(S, \oS)|$, and $\min\{|S'|, |\oS'|\} = \min\{|S|, |\oS|\} + 1$. Thus,
\begin{align*}
\iota_G(S') = \frac{|E(S', \oS')|}{\min\{|S'|, |\oS'|\}} \leq \frac{|E(S, \oS)|}{\min\{|S|, |\oS|\} + 1} < \iota_G(S).
\end{align*}
\end{itemize}
In both cases, we have $\iota_G(S') < \iota_G(S)$, contradicting the optimality of $S$.
\end{proof}

\subsection{Scoring-Based Tie-Breaking}
\label{sec:lexicographic-tie-breaking}

To introduce scoring-based tie-breaking, it will be best to think of having a (random) partial order $\succ$ on all subsets of $[n]$ (which will be specified later).
We say that a set $S$ is a \emph{minimal optimal set} with respect to a partial ordering $\succ$ on $[n]$ if, among the optimal sets, it is minimal according to $\succ$.

For a given permutation $\pi: [n] \to [n]$, we define the \emph{score} of $S$ as $\scr(S) := \sum_{v \in S} \pi(v)$. Let $\succ_{\pi}$ be the lexicographic ordering of the tuple $(|S|, \scr(S))$. That is, $S' \succeq_{\pi} S$ iff either $|S'| > |S|$, or $|S'| = |S|$ and $\scr(S') \geq \scr(S)$. We use $\cO_{\pi}(G)$ to denote the collection of all minimal optimal sets of $G$ under $\pi$.

The main new lemma we have is the following:
\begin{lemma} \label{lem:lex}
Let $G = ([n], E)$ be any $\Delta$-regular graph and $S \subseteq [n]$ be any optimal set for $G$ with $(\bs, \obs) = \confv_G(S)$.
If $\pi: [n] \to [n]$ is a random permutation, then
\begin{align*}
\Pr_{\pi}[S \in \cO_{\pi}(G)] \leq 
n^{2\Delta} \cdot \frac{1}{\binom{s_{\Delta/2} + \os_{\Delta/2}}{s_{\Delta/2}}}.
\end{align*}
\end{lemma}

Before we prove the above lemma, it will be convenient to start by giving a simple necessary condition for $S$ to be in $\cO_{\pi}(G)$, as stated below\footnote{Recall from \Cref{subsec:prelim-config} that $S_{\Delta/2}$ (resp. $\oS_{\Delta/2}$) denote the set of vertices in $S$ (resp. $\oS$) with cross-degree $\Delta/2$.}.

\begin{lemma} \label{lem:opt-characterization}
%For any $\Delta$-regular graph $G$, $S \subseteq [n]$ of size at most $n/2$ and
Let $G, S, (\bs, \obs)$ be as in \Cref{lem:lex}.
For any $\pi: [n] \to [n]$, if $S \in \cO_{\pi}(G)$ and $s_{\Delta/2}, \os_{\Delta/2} \geq \Delta+1$, then there exists $T \subseteq S_{\Delta/2}$ of size $\Delta$ such that $\pi(v) < \pi(\ov)$ for all $v \in S_{\Delta/2} \setminus T, \ov \in \oS_{\Delta/2}$.
\end{lemma}

\begin{proof}
$\ov^*$ be the set of vertex in $\oS$ with smallest value with respect to $\pi$, i.e. $\ov^* = \argmin_{\ov \in \oS} \pi(\ov)$. Then, let $T^*$ be any $\Delta$-size subset of $S_{\Delta/2}$ that contains all neighbors of $\ov^*$. We claim that $T^*$ satisfies the property. Suppose that this is false. Then, there exists some $v^* \in S_{\Delta/2} \setminus T$ such that $\pi(v^*) > \pi(\ov^*)$. Now, consider $S' = S \setminus \{v^*\} \cup \{\ov^*\}$. Since $v^* \notin T$, it is not a neighbor of $\ov^*$. From this and from $v^* \in S_{\Delta/2}, \ov^* \in \oS_{\Delta/2}$, we have $|E(S', \oS')| = |E(S, \oS)|$. Thus, $\iota_G(S) = \iota_G(S')$. Meanwhile, we have $\scr(S') = \scr(S) + \pi(\ov^*) - \pi(v^*) < \scr(S)$. This implies that $S \succneqq_{\pi} S'$, which contradicts $S \in \cO_\pi(G)$.
\end{proof}

We can now easily prove \Cref{lem:lex}.

\begin{proof}[Proof of \Cref{lem:lex}]
If $s_{\Delta/2} \leq \Delta$ or $\os_{\Delta/2} \leq \Delta$, then the bound holds trivially as the RHS is more than one. Note also that if $|S| > n/2$, then $\Pr_{\pi}[S \in \cO_{\pi}(G)] = 0$ since $S \succneqq_{\pi} \oS$. Thus, we can suppose henceforth that $s_{\Delta/2}, \os_{\Delta/2} \geq \Delta + 1$ and $|S| \leq n/2$. Applying \Cref{lem:opt-characterization} and the union bound yields
\begin{align*}
\Pr_{\pi}[S \in \cO_\pi(G)] &\leq \Pr_{\pi}\left[\exists T \in \binom{S_{\Delta/2}}{\Delta}, \pi(v) < \pi(\ov) \text{ for all } v \in S_{\Delta/2} \setminus T, \ov \in \oS_{\Delta/2}\right] \\
&\leq \sum_{T \in \binom{S_{\Delta/2}}{\Delta}} \Pr_{\pi}\left[\pi(v) < \pi(\ov) \text{ for all } v \in S_{\Delta/2} \setminus T, \ov \in \oS_{\Delta/2}\right].
\end{align*}
Since $\pi$ is a uniformly random permutation, the relative order of elements in $(S_{\Delta/2} \setminus T) \cup \oS_{\Delta/2}$ is uniformly random. Thus, the inner probability, which is the probability that $\pi$ orders all elements in $(S_{\Delta/2} \setminus T)$ before $\oS_{\Delta/2}$, is exactly $\binom{s_{\Delta/2} + \os_{\Delta/2} - \Delta}{\os_{\Delta/2}}$. Plugging this into the above gives
\begin{align*}
\Pr_{\pi}[S \in \cO_\pi(G)] \leq n^{\Delta} \cdot \frac{1}{\binom{s_{\Delta/2} + \os_{\Delta/2} - \Delta}{\os_{\Delta/2}}}
\leq n^{2\Delta} \cdot  \frac{1}{\binom{s_{\Delta/2} + \os_{\Delta/2}}{\os_{\Delta/2}}}. &\qedhere
\end{align*}
\end{proof}

\section{Improved Bound for Expansion: Proof of \Cref{thm:main}}

We are now ready to prove our main theorem (\Cref{thm:main}).
Let $\unu = \unu_\Delta$ and $\ualpha$ be from \Cref{thm:bollobas-classic} and \Cref{thm:small-set} respectively.

We have
\begin{align*}
\Pr_{G \sim \cP(n, \Delta)}[\iota(G) < \nu] &\leq \Pr_{G \sim \cP(n, \Delta)}[\iota(G) < \unu] + \Pr_{G \sim \cP(n, \Delta)}[\iota_{\leq \ualpha}(G) < \nu] + \Pr_{G \sim \cP(n, \Delta)}[\unu \leq \iota(G) < \nu \wedge \iota_{\leq \ualpha}(G) \geq \nu] \\
&= o(1) + \Pr_{G \sim \cP(n, \Delta)}[\unu \leq \iota(G) < \nu \wedge \iota_{\leq \ualpha}(G) \geq \nu].
\end{align*}
where the last inequality is due to \Cref{thm:bollobas-classic,thm:small-set}. We can thus focus on showing that the last term is $o(1)$.
We further write this term as follows, where $\pi$ is a uniformly random permutation (independent of $G$).
\begin{align*}
&\Pr_{G \sim \cP(n, \Delta)}[\unu \leq \iota(G) < \nu \wedge \iota_{\leq \ualpha}(G) \geq \nu] \\
&= \Pr_{G \sim \cP(n, \Delta), \pi}\left[\exists S, \ualpha n < |S| \leq n/2 \text{ and } \unu \leq \iota_{G}(S) < \nu \text{ and } S \in \cO_{\pi}(G)\right] \\
&= \sum_{k=\lfloor \ualpha n \rfloor + 1}^{n/2} \sum_{c = \lceil \unu k \rceil}^{\lceil \nu k \rceil -1} \sum_{(\bs, \obs') \in \bbUl(k, c)} \sum_{S \in \binom{[n]}{k}} \Pr_{G, \pi}[\confv_{G}(S) = (\bs, \obs) \text{ and } S \in \cO_\pi(G)] \\
&= \sum_{k=\lfloor \ualpha n \rfloor + 1}^{n/2} \sum_{c = \lceil \unu k \rceil}^{\lceil \nu k \rceil -1} \sum_{(\bs, \obs') \in \bbUl(k, c)} \sum_{S \in \binom{[n]}{k}} \Pr_{G}[\confv_{G}(S) = (\bs, \obs)]\Pr_{G, \pi}[S \in \cO_\pi(G) \mid \confv_{G}(S) = (\bs, \obs)] \\
&\leq \sum_{k=\lfloor \ualpha n \rfloor + 1}^{n/2} \sum_{c = \lceil \unu k \rceil}^{\lceil \nu k \rceil -1} \sum_{(\bs, \obs') \in \bbUl(k, c)} \binom{n}{k} \cdot P(\bs, \obs) \cdot n^{2\Delta} \cdot \frac{1}{\binom{s_{\Delta/2} + \os_{\Delta/2}}{s_{\Delta/2}}} \\
&\leq 2^{o(n)} \cdot \max_{k, c, (\bs, \obs') \in \bbUl(k, c)} \binom{n}{k} \cdot P(\bs, \obs) \cdot \frac{1}{\binom{s_{\Delta/2} + \os_{\Delta/2}}{s_{\Delta/2}}}
\end{align*}
where the first inequality is from \Cref{lem:lampis-counting,lem:lex}, and the last maximum is over $k \in \{\lfloor \ualpha n \rfloor + 1, \dots, n/2\}$ and $c \in \{\lceil \unu k \rceil, \dots, \lceil \nu k \rceil - 1\}$.

As such, it suffices to show that the inner most term is $2^{-\Omega(n)}$. For $(\bs, \obs) \in \bbUl(k, c)$, we can simplify this term further as follows:
\begin{align*}
&F(n, k, c, \bs, \obs) := \binom{n}{k} \cdot P(\bs, \obs) \cdot \frac{1}{\binom{s_{\Delta/2} + \os_{\Delta/2}}{s_{\Delta/2}}} \\
&= \frac{n!}{(\Delta n)!!} \cdot \left(c! (\Delta k - c)!! (\Delta(n- k) - c)!!\right) \cdot \left(\prod_{i=0}^{\Delta/2} \frac{1}{s_i!} \cdot \binom{\Delta}{i}^{s_i} \right) \cdot \left(\prod_{i=0}^{\Delta/2} \frac{1}{\os_i!} \cdot \binom{\Delta}{i}^{\os_i} \right) \cdot \frac{1}{\binom{s_{\Delta/2} + \os_{\Delta/2}}{s_{\Delta/2}}} \\
&= \frac{n!}{(\Delta n)!!} \cdot \left(c! (\Delta k - c)!! (\Delta(n- k) - c)!!\right) \\ &\qquad \cdot \left(\frac{1}{(s_{\Delta/2} + \os_{\Delta/2})!} \binom{\Delta}{{\Delta/2}}^{s_{\Delta/2} + \os_{\Delta/2}}\right) \cdot \left(\prod_{i=0}^{{\Delta/2}-1} \frac{1}{s_i!} \cdot \binom{\Delta}{i}^{s_i} \right) \cdot \left(\prod_{i=0}^{{\Delta/2}-1} \frac{1}{\os_i!} \cdot \binom{\Delta}{i}^{\os_i} \right).
\end{align*}

We aim to find an asymptotic estimate for $F(n, k, c, \obs, \bs)$. Let $\alpha = k/n, \gamma = c/n, \bx = \bs / n$ and $\obx = \obs / n$.  
Using Stirling's approximation (from \Cref{subsec:stirling}), we get
\begin{align*}
\frac{1}{n} \ln F(n, k, c, \obs, \bs) - o(1) &\leq f(\alpha, \gamma, \bx, \obx) \\ &:= -\frac{\Delta}{2} \ln \Delta + \gamma \ln \gamma + \frac{\alpha\Delta - \gamma}{2} \ln\big(\alpha\Delta - \gamma\big) + \frac{\Delta(1-\alpha) - \gamma}{2} \ln\big(\Delta(1-\alpha) - \gamma\big) \\
&+ (x_{\Delta/2} + \ox_{\Delta/2}) \ln\Paren{\frac{\binom{\Delta}{\Delta/2}}{x_{\Delta/2} + \ox_{\Delta/2}}} + \sum_{i=0}^{\Delta/2 - 1} x_i \ln\Paren{\frac{\binom{\Delta}{i}}{x_i}} + \sum_{i=0}^{\Delta/2-1} \ox_i \ln \Paren{\frac{\binom{\Delta}{i}}{\ox_i}}
\end{align*}

Also defined the normalized version of (each vector in) $\bbUl(n, k)$ as follows: Let $\bbXl(\alpha, \gamma)$ denote the set of all $\bx \in [0,1]^{[\Delta/2]_0}$ that satisfies the following constraints:
\begin{itemize}
\item $x_j = 0$ for all $j > \frac{\Delta}{2}$,
\item Furthermore, if $\alpha > 1/2$, it must be that $x_{\frac{\Delta}{2}} = 0$.
\item $\sum_{i \in [\Delta/2]_0} x_i = \alpha$
\item $\sum_{i \in [\Delta/2]_0} i \cdot x_i = \gamma$ 
\end{itemize}
By definition, if $(\bs, \obs) \in \bbUl(n, k)$, then $\bx \in \bbXl(\alpha, \gamma)$ and $ \obx \in \bbXl(1 - \alpha, \gamma)$. Thus, combining all the bounds so far, we have
\begin{align*}
\Pr_{G \sim \cP(n, \Delta)}[\iota(G) < \nu] &\leq o(1) + 2^{o(n)} \cdot 2^{n \cdot \sup_{\alpha \in (\ualpha, 1/2]} \sup_{\gamma \in [\alpha \cdot \unu, \alpha \cdot \nu]} \sup_{\bx \in \bbXl(\alpha, \gamma) \atop \obx \in \bbXl(1 - \alpha, \gamma)} f(\alpha, \gamma, \bx, \obx)}
\end{align*}

Thus, it suffices to show that the supremum term in the exponent on the RHS is negative, by considering two cases based on whether $\alpha = 1/2$.

\subsection{Symmetric Case: $\alpha = 1/2$}

Notice that
\begin{align*}
f(1/2, \gamma, \bx, \obx) = &-\frac{\Delta}{2} \ln \Delta + \gamma \ln \gamma + \Paren{\Delta/2 - \gamma} \ln\big(\Delta/2 - \gamma\big) \\
&+ (x_{\Delta/2} + \ox_{\Delta/2}) \ln\Paren{\frac{\binom{\Delta}{\Delta/2}}{x_{\Delta/2} + \ox_{\Delta/2}}} + \sum_{i=0}^{\Delta/2 - 1} x_i \ln\Paren{\frac{\binom{\Delta}{i}}{x_i}} + \sum_{i=0}^{\Delta/2-1} \ox_i \ln \Paren{\frac{\binom{\Delta}{i}}{\ox_i}}
\end{align*}

For any $\bx, \obx$, let $\tbx = (\bx + \obx)/2$. Since $x \ln(x)$ is convex, we have
\begin{align*}
f(1/2, \gamma, \bx, \obx) \leq &-\frac{\Delta}{2} \ln \Delta + \gamma \ln \gamma + \Paren{\Delta/2 - \gamma} \ln\big(\Delta/2 - \gamma\big) \\
&+ 2\tx_{\Delta/2} \ln\Paren{\frac{\binom{\Delta}{\Delta/2}/2}{\tx_{\Delta/2}}} + 2\sum_{i=0}^{\Delta/2 - 1} \tx_i \ln\Paren{\frac{\binom{\Delta}{i}}{\tx_i}} &=: h(\gamma, \tbx).
\end{align*}
Notice that, when $\bx, \obx \in (\bx, \obx) \in \bbXl(1/2, \gamma)$, we have that $\tbx \in \bbXl(1/2, \gamma)$. As a result, we can conclude that
\begin{align*}
\sup_{\bx, \obx \in \bbXl(1/2, \gamma)} f(1/2, \gamma, \bx, \obx) \leq \sup_{\tbx \in \bbXl(1/2, \gamma)} h(\gamma, \tbx) =: H_{\Delta}(\gamma).
\end{align*}

Let us define $\bb \in \R^{[\Delta/2]_0}$ by 
\begin{align*}
b_i =
\begin{cases}
\binom{\Delta}{i} & \text{ if } i \ne \Delta/2, \\
\binom{\Delta}{\Delta/2}/2 & \text{ if } i = \Delta/2.
\end{cases}
\end{align*}
Then, we have\footnote{Recall that $\Phi^*$ is defined in \Cref{opt-prob:entropy-min}.}
\begin{align} \label{eq:def-for-main-bound}
H_{\Delta}(\gamma) = -\frac{\Delta}{2} \ln \Delta + \gamma \ln \gamma + \Paren{\Delta/2 - \gamma} \ln\big(\Delta/2 - \gamma\big) + 2 \Phi^*(1/2, \gamma, \bb).
\end{align}
We henceforth drop the subscript $\Delta$ for brevity.
From \Cref{thm:min-entropy-main}, we now have
\begin{align} \label{eq:symmetric-h-derivative}
H'(\gamma) = \ln\left(\frac{\gamma}{\frac{\Delta}{2} - \gamma}\right) - 2\ln z^*
\end{align}
where $z^*$ is the solution to
\begin{align} \label{eq:polynomial-eq}
P(z) := \sum_{i=0}^{\Delta/2} \Paren{\frac{i}{2} - \gamma} b_i z^i = 0.
\end{align}

We will now prove the following technical lemma that bounds $z^*$.

\begin{lemma} \label{lem:root-bound}
Let $\gamma \in (0, \Delta/8)$, for $z^*$ that is a root of $P(z)$ as defined in \eqref{eq:polynomial-eq}, we have $z^* < \sqrt{\frac{\gamma}{\frac{\Delta}{2} - \gamma}}$.
\end{lemma}

\begin{proof}
There is exactly one sign change in the coefficients; by Descartes' Rule of Signs, $P(z)$ has exactly one positive real root, $z^*$. Since $P(0) < 0$ and $P(z) \to \infty$ as $z \to \infty$, to show that $z^* < \sqrt{\frac{\gamma}{\frac{\Delta}{2} - \gamma}} =: z_0$ it suffices to show that $P(z_0) > 0$.

Since $z_0 = \sqrt{\frac{\gamma}{\Delta/2 - \gamma}}$, rewriting this relation yields $\gamma = \frac{\frac{\Delta}{2} z_0^2}{1 + z_0^2}.$ Thus, we have
\begin{align*}
P(z_0) > 0 \Leftrightarrow (1 + z_0^2)P(z_0) > 0
\Leftrightarrow (1 + z_0^2) S_1(z_0) - \frac{\Delta}{2} z_0^2 S_2(z_0) > 0
\end{align*}
where $S_1(z) := \sum_{i=0}^{\Delta/2} b_i \frac{i}{2} z^i$ and $S_2(z) := \sum_{i=0}^{\Delta/2} b_i z^i.$
Let $Q(z) = (1 + z^2) S_1(z) - \frac{\Delta}{2} z^2 S_2(z).$ Our goal is now to show that $Q(z_0) > 0$. Since $\gamma < \Delta/8$, we have $z_0 = \sqrt{\frac{\gamma}{\Delta/2 - \gamma}} = \frac{1}{\sqrt{\Delta/(2\gamma) - 1}} < \frac{1}{\sqrt{3}}$. Thus, it suffices to show that $Q(z) > 0$ for all $z \in (0, 1/\sqrt{3})$. To do this, we analyze the three cases for $\Delta = 4, 6, 8$.

\paragraph{Case 1: $\Delta = 4$.}
Here, $\Delta/2 = 2$. The coefficients are $b_0 = 1$, $b_1 = 4$, $b_2 = 3$.
We have
\begin{align*}
S_1(z) = 2z + 3z^2, & &S_2(z) = 1 + 4z + 3z^2.
\end{align*}
Evaluating $Q(z)$ with $\frac{\Delta}{2} = 2$:
\begin{align*}
Q(z) &= (1+z^2)(2z + 3z^2) - 2z^2(1 + 4z + 3z^2) \\
&= (2z + 3z^2 + 2z^3 + 3z^4) - (2z^2 + 8z^3 + 6z^4) \\
&= 2z + z^2 - 6z^3 - 3z^4 \\
&= z(2+z)(1 - 3z^2).
\end{align*}
For $z \in (0, 1/\sqrt{3})$, $1 - 3z^2 > 0$. Hence, all factors are non-negative, yielding $Q(z) \geq 0$.

\paragraph{Case 2: $\Delta = 6$.}
Here, $\Delta/2 = 3$. The coefficients are $b_0 = 1, b_1 = 6, b_2 = 15, b_3 = 10$.
We have
\begin{align*}
S_1(z) = 3z + 15z^2 + 15z^3,
& &S_2(z) = 1 + 6z + 15z^2 + 10z^3.
\end{align*}
Evaluating $Q(z)$ with $\frac{\Delta}{2} = 3$:
\begin{align*}
Q(z) &= (1+z^2)(3z + 15z^2 + 15z^3) - 3z^2(1 + 6z + 15z^2 + 10z^3) \\
&= (3z + 15z^2 + 15z^3 + 3z^3 + 15z^4 + 15z^5) - (3z^2 + 18z^3 + 45z^4 + 30z^5) \\
&= 3z + 12z^2 - 30z^4 - 15z^5 \\
&= 3z\Paren{(1 - 5z^4) + z(4 - 10z^2)}.
\end{align*}
Similarly, all terms are positive for $z \in (0, 1/\sqrt{3})$, implying $Q(z) > 0$.

\paragraph{Case 3: $\Delta = 8$.}
Here, $\Delta/2 = 4$. The coefficients are $b_0 = 1, b_1 = 8, b_2 = 28, b_3 = 56, b_4 = 35$.
We have
\begin{align*}
S_1(z) = 4z + 28z^2 + 84z^3 + 70z^4, 
& &S_2(z) = 1 + 8z + 28z^2 + 56z^3 + 35z^4.
\end{align*}
Evaluating $Q(z)$ with $\frac{\Delta}{2} = 4$:
\begin{align*}
Q(z) &= (1+z^2)(4z + 28z^2 + 84z^3 + 70z^4) - 4z^2(1 + 8z + 28z^2 + 56z^3 + 35z^4) \\
&= 4z + 28z^2 + 84z^3 + 70z^4 + 4z^3 + 28z^4 + 84z^5 + 70z^6 \\
&\quad - (4z^2 + 32z^3 + 112z^4 + 224z^5 + 140z^6) \\
&= 4z + 24z^2 + 56z^3 - 14z^4 - 140z^5 - 70z^6 \\
&= 2z\Paren{2 + z(3 - 7z^2) + z(9 - 35z^4) + 14z^2(2 - 5z^2)}.
\end{align*}
Again, all terms are positive for $z \in (0, 1/\sqrt{3})$, implying $Q(z) > 0$.
\end{proof}

From \eqref{eq:symmetric-h-derivative} and \Cref{lem:root-bound}, $H$ is strictly increasing in $\gamma$ (for $\gamma \in (0, \Delta/8)$). Thus, we have
\begin{align*}
\sup_{\gamma \in [0, \nu/2]} \sup_{\bx, \obx \in \bbXl(1/2, \gamma)} f(1/2, \gamma, \bx, \obx) \leq \sup_{\gamma \in [0, \nu/2]} H(\gamma) = H_{\Delta}(\nu/2) < H(\nu^*_{\Delta}/2) = 0,
\end{align*}
where the last equality is from our definition of $\nu^*_{\Delta}$ (in \Cref{thm:main}). This concludes the proof for $\alpha = 1/2$.

\subsection{Asymmetric Case: $\alpha \ne 1/2$}
\label{subsec:asym}

Recall that our goal is to show that
\begin{align*}
f^* := \sup_{\alpha \in (\ualpha, 1/2)} \sup_{\gamma \in [\alpha \cdot \unu, \alpha \cdot \nu]} \sup_{\bx \in \bbXl(\alpha, \gamma) \atop \obx \in \bbXl(1 - \alpha, \gamma)} f(\alpha, \gamma, \bx, \obx)
\end{align*}
is negative.

To do so, first recall that in this case $\ox_{\Delta/2} = 0$. Thus, we have
\begin{align*}
\sup_{\bx \in \bbXl(\alpha, \gamma) \atop \obx \in \bbXl(1 - \alpha, \gamma)} f(\alpha, \gamma, \bx, \obx) 
&=  -\frac{\Delta}{2} \ln \Delta + \gamma \ln \gamma + \frac{\alpha\Delta - \gamma}{2} \ln\big(\alpha\Delta - \gamma\big) + \frac{\Delta(1-\alpha) - \gamma}{2} \ln\big(\Delta(1-\alpha) - \gamma\big) \\
&\qquad + \sup_{\bx \in \bbXl(\alpha, \gamma)} \Paren{\sum_{i=0}^{\Delta/2} x_i \ln\Paren{\frac{\binom{\Delta}{i}}{x_i}}} + \sup_{\obx \in \bbXl(1 - \alpha, \gamma)} \Paren{ \sum_{i=0}^{\Delta/2-1} \ox_i \ln \Paren{\frac{\binom{\Delta}{i}}{\ox_i}}}.
\end{align*}
Let us define $\bb \in \R^{[\Delta/2]_0}$ by $b_i = \binom{\Delta}{i}$ and $\obb \in \R^{[\Delta/2-1]_0}$ by $\ob_i = \binom{\Delta}{i}$. We can simply write the above as
\begin{align*}
\sup_{\bx \in \bbXl(\alpha, \gamma) \atop \obx \in \bbXl(1 - \alpha, \gamma)} f(\alpha, \gamma, \bx, \obx) &= -\frac{\Delta}{2} \ln \Delta + \gamma \ln \gamma + \frac{\alpha\Delta - \gamma}{2} \ln\big(\alpha\Delta - \gamma\big) + \frac{\Delta(1-\alpha) - \gamma}{2} \ln\big(\Delta(1-\alpha) - \gamma\big) \\
&\qquad + \Phi^*(\alpha, \gamma, \bb) + \Phi^*(1 - \alpha, \gamma, \obb).
\end{align*}
Let $G(\alpha, \gamma) = -\frac{\Delta}{2} \ln \Delta + \gamma \ln \gamma + \frac{\alpha\Delta - \gamma}{2} \ln\big(\alpha\Delta - \gamma\big) + \frac{\Delta(1-\alpha) - \gamma}{2} \ln\big(\Delta(1-\alpha) - \gamma\big)$ denote the first term above.

We then certify that $f^* > 0$ using a numerical approach via discretization. In particular, we let $\cI^{\alpha}$ denote the set of intervals $[\alphalo, \alphaup]$ such that $(\ualpha, 1/2) \subseteq \bigcup_{[\alphalo, \alphaup] \in \cI^{\alpha}} [\alphalo, \alphaup]$ and, similarly, let $\cI^{\gamma}_{\alphalo, \alphaup}$ denote the set of intervals $[\gammalo, \gammaup]$ such that $[\alphalo \cdot \unu, \alphaup \cdot \nu] \subseteq \bigcup_{[\gammalo, \gammaup] \in \cI^{\gamma}_{\alphalo, \alphaup}} [\gammalo, \gammaup]$. We have that
\begin{align}
f^* &= \sup_{\alpha \in (\ualpha, 1/2)} \sup_{\gamma \in [\alpha \cdot \unu, \alpha \cdot \nu]} G(\alpha, \gamma) + \Phi^*(\alpha, \gamma, \bb) + \Phi^*(1 - \alpha, \gamma, \obb) \nonumber \\
&\leq \max_{[\alphalo, \alphaup] \in \cI^{\alpha}} \max_{[\gammalo, \gammaup] \in \cI^{\gamma}_{\alphalo, \alphaup}} \sup_{\alpha \in [\alphalo, \alphaup], \gamma \in [\gammalo, \gammaup]} \Paren{G(\alpha, \gamma) + \Phi^*(\alpha, \gamma, \bb) + \Phi^*(1 - \alpha, \gamma, \obb)} \label{eq:f-star-expanded}
\end{align}
We note that 
$$\frac{\partial G}{\partial \alpha} = \frac{\Delta}{2} \ln \left( \frac{\alpha\Delta - \gamma}{\Delta(1-\alpha) - \gamma} \right)$$
is negative since $\alpha \in (0,1/2)$, and
$$\frac{\partial G}{\partial \gamma} = \ln \left( \frac{\gamma}{\sqrt{(\alpha\Delta - \gamma)(\Delta(1-\alpha) - \gamma)}} \right)$$ is also negative since $\gamma < \Delta \alpha / 4$. In other words, $G$ is a decreasing function of both $\alpha$ and $\gamma$. Thus, we have
\begin{align*}
\sup_{\alpha \in [\alphalo, \alphaup], \gamma \in [\gammalo, \gammaup]} G(\alpha, \gamma) \leq G(\alphalo, \gammalo).
\end{align*}

Meanwhile, $\Phi^*(\alpha, \gamma, \bb) + \Phi^*(1 - \alpha, \gamma, \obb)$ are from optimization problem \eqref{opt-prob:entropy-min}, which is a maximization of a concave objective under linear constraints. This implies that the supremum is simply the maximum at one of the corners. More formally, we have
\begin{align*}
\sup_{\alpha \in [\alphalo, \alphaup], \gamma \in [\gammalo, \gammaup]} \Phi^*(\alpha, \gamma, \bb) &\leq \max_{\alpha \in \{\alphalo, \alphaup\}, \gamma \in \{\gammalo, \gammaup\}} \Phi^*(\alpha, \gamma, \bb), \\
\sup_{\alpha \in [\alphalo, \alphaup], \gamma \in [\gammalo, \gammaup]} \Phi^*(1 - \alpha, \gamma, \bb) &\leq \max_{\alpha \in \{\alphalo, \alphaup\}, \gamma \in \{\gammalo, \gammaup\}} \Phi^*(1 - \alpha, \gamma, \bb).
\end{align*}

Plugging these into \eqref{eq:f-star-expanded}, we can upper bound $f^*$ by
\begin{align*}
&\max_{[\alphalo, \alphaup] \in \cI^{\alpha}} \max_{[\gammalo, \gammaup] \in \cI^{\gamma}_{\alphalo, \alphaup}} \sup_{\alpha \in [\alphalo, \alphaup], \gamma \in [\gammalo, \gammaup]} \Paren{G(\alpha, \gamma) + \Phi^*(\alpha, \gamma, \bb) + \Phi^*(1 - \alpha, \gamma, \obb)} \\
&\leq \max_{[\alphalo, \alphaup] \in \cI^{\alpha}} \max_{[\gammalo, \gammaup] \in \cI^{\gamma}_{\alphalo, \alphaup}} \Paren{G(\alphalo, \gammalo) + \max_{\alpha \in \{\alphalo, \alphaup\}, \gamma \in \{\gammalo, \gammaup\}} \Phi^*(\alpha, \gamma, \bb) + \max_{\alpha \in \{\alphalo, \alphaup\}, \gamma \in \{\gammalo, \gammaup\}} \Phi^*(1 - \alpha, \gamma, \bb)}.
\end{align*}
To certify $f^* < 0$, we simply compute the inner term above and take the maximum. We use linear discretization for $\alpha, \gamma$ where we divide each interval $(\ualpha, 1/2)$ and $[\alphalo \cdot \unu, \alphaup \cdot \nu]$ to $M$ subintervals of equal sizes.
Using $M = 200$, $\ualpha = 0.1$ from \Cref{thm:small-set}, and $\unu$ from \Cref{thm:bollobas-classic}, this numerical approach certifies that $f^*$ is at most $-0.009, -0.004$ and $-0.001$ for $\Delta = 4, 6, 8$ respectively. This concludes our proof.

\begin{remark}
We end this section by remarking that the reason that a discretization-based upper bound for $f^*$ works here is that $\bbXl(1 - \alpha, \gamma)$ is \emph{more restrictive} for the asymmetric case, since $\ox_{\Delta/2}$ must be zero here whereas this is allowed to be non-zero in the symmetric case. This creates a gap between the two cases, allowing us to use a simple discretization.
\end{remark}
\section{Conclusion and Research Direction}

In this work, we use a scoring-based tie breaking technique to provide improved lower bounds on expansions of random regular graphs with small even degrees. Our approach can be extended to higher degree; however, obtaining a formal bound becomes more challenging since formal inequalities (e.g. \Cref{lem:root-bound}) are harder to prove. Furthermore, based on our numerical approximations, it also seems that the improvement diminishes as $\Delta$ increases. Another interesting direction is to try to apply the technique to the case $\Delta = 3$. As mentioned earlier, Kostochka and Melnikov's bound~\cite{KM92} remains the best known lower bounds to date. In their proof, there is already a scoring-based tie-breaking in the form of ``marks'' between edge-vertex pairs. However, given that the marks are bounded (to be in $-1, 0, 1$), it seems plausible that a more sophisticated scoring-based tie-breaking may help. While this seems like a promising direction, we note that it is quite challenging since Kostochka and Melnikov's proof relies on a more complicated counting argument compared to the one used in our work (\Cref{lem:lampis-counting}).

\paragraph*{Acknowledgment.}
I am grateful to Ansh Nagda for discussions that eventually led to this work.

\bibliography{ref}
\bibliographystyle{alpha}

\end{document}